\documentclass[11pt]{article}
\usepackage[T1]{fontenc}
\usepackage[utf8]{inputenc}
\usepackage[english]{babel}
\usepackage{amsmath}
\usepackage{amssymb}
\usepackage{amsfonts}
\usepackage{amsthm}
\usepackage{mathtools}
\usepackage{hyperref}
\usepackage[nameinlink,noabbrev]{cleveref}
\usepackage{graphicx}
\usepackage{tikz}
\usepackage{booktabs}
\usepackage{enumitem}
\usepackage[margin=1in]{geometry}
\hypersetup{
    colorlinks=true,
    linkcolor=blue,
    citecolor=blue,
    urlcolor=blue
}

\theoremstyle{plain}

\newtheorem{theorem}{Theorem}[section]
\newtheorem{lemma}[theorem]{Lemma}
\newtheorem{proposition}[theorem]{Proposition}
\newtheorem{corollary}[theorem]{Corollary}

\theoremstyle{definition}

\newtheorem{definition}[theorem]{Definition}

\theoremstyle{remark}

\Crefname{theorem}{theorem}{theorems}
\Crefname{theorem}{Theorem}{Theorems}

\Crefname{lemma}{lemma}{lemmas}
\Crefname{lemma}{Lemma}{Lemmas}

\Crefname{proposition}{proposition}{propositions}
\Crefname{proposition}{Proposition}{Propositions}

\Crefname{corollary}{corollary}{corollaries}
\Crefname{corollary}{Corollary}{Corollaries}

\Crefname{definition}{definition}{definitions}
\Crefname{definition}{Definition}{Definitions}

\Crefname{example}{example}{examples}
\Crefname{example}{Example}{Examples}

\Crefname{remark}{remark}{remarks}
\Crefname{remark}{Remark}{Remarks}

\Crefname{conjecture}{conjecture}{conjectures}
\Crefname{conjecture}{Conjecture}{Conjectures}

\Crefname{claim}{claim}{claims}
\Crefname{claim}{Claim}{Claims}

\newcommand{\bb}[1]{\mathbb{#1}}
\DeclareMathOperator\supp{supp}

\title{Orthogonal partition by four hyperplanes}
\author{Alexey Fakhrutdinov \and Platon Khamatov \and Oleg R. Musin}
\date{}
\begin{document}
\maketitle
\numberwithin{equation}{section}

\begin{abstract}
We study orthogonal $(4,3)$-equipartitions of finite measures absolutely continuous with respect to Lebesgue measure: every three of four mutually
orthogonal affine hyperplanes must divide each measure into eight
equal parts. For every number $m$ of measures,
we determine the exact dimension threshold of Musin's polynomial
criterion and express it in terms of the binary expansion of $m$.
The threshold coincides with the geometric lower bound precisely
when $m=2^j-1$. Consequently,
$\Delta^*(2^j-1,4,3)=7\cdot2^{j-1}-2$ for every $j\geq1$.
\end{abstract}

\section{Introduction}
Let $m\geq1$ and $1\leq n\leq k$. We consider finite Borel
measures on $\bb{R}^d$ that assign zero measure to every affine
hyperplane. A \emph{$(k,n)$-equipartition} is a collection of
$k$ affine hyperplanes such that every $n$-element subfamily
divides each measure into $2^n$ equal parts. It is called
\emph{orthogonal} if the hyperplanes are mutually orthogonal.
Let $\Delta(m,k,n)$ and $\Delta^*(m,k,n)$ denote the least
dimensions in which every collection of $m$ such measures
admits these respective equipartitions; clearly,
$\Delta(m,k,n)\leq\Delta^*(m,k,n)$.

The non-orthogonal case $n=k$ is the classical
Gr\"unbaum--Hadwiger--Ramos problem
\cite{BFHZ2018,MVZ2006,Ramos1996}.
Gr\"unbaum asked whether $d$ hyperplanes always divide a
mass in $\bb{R}^d$ into $2^d$ equal parts. This holds for
$d\leq3$ and fails for $d\geq5$ by Avis \cite{Avis1984}.
Sober\'on \cite{Soberon2026} recently settled the remaining
case $d=4$ by constructing a smooth strictly positive density
that no four affine hyperplanes divide into $16$ equal parts.
Together with the upper bound in \cite{BFHZ2018}, this gives
$\Delta(1,4,4)=5$.

Partial equipartitions originate in Makeev's work on
orthogonal pairwise partitions \cite{Makeev2007} and were
formulated for general $(k,n)$ and several measures by
Blagojevi\'c and Karasev \cite{BK2012}.
Mejia, Simon, and Zhang \cite{MSZ2025} studied $n=2$ and
$(k,n)=(4,3)$; related constraints were considered by
Simon \cite{Simon2019}. See also \cite{RPS2022,Matousek2003}.

For $k=4$, the four dimension problems ask for the least
$d$ in which every collection of $m$ measures admits four
mutually orthogonal hyperplanes satisfying:
\begin{enumerate}[label=\textbf{(4,\arabic*)},leftmargin=*,nosep]
  \item each hyperplane bisects every measure;
  \item every pair divides every measure into four equal parts;
  \item every triple divides every measure into eight equal parts;
  \item all four divide every measure into sixteen equal parts.
\end{enumerate}

Musin \cite[Theorems~1.4--1.6 and Section~4.1]{Musin2026}
proved $\Delta^*(m,k,1)=m+k-1$ and established general
dimension bounds for the families $(k,2)$ and $(k,k)$,
including exact values for $(k,2)$ when $m=2^j-1$.
Thus $(4,3)$ is the only case for four hyperplanes outside
these general families. Although individual values and
powers of two were treated in \cite{Musin2026}, arbitrary
$m$ requires a separate analysis of monomial cancellations.
Here we determine the exact threshold of the polynomial
criterion for every $m$ and obtain exact geometric
dimensions for $m=2^j-1$.

For full equipartitions without the orthogonality requirement,
Ramos \cite{Ramos1996} proved the lower bound
$\lceil m(2^k-1)/k\rceil$.
Musin \cite{Musin2026} established the following extension
to orthogonal $(k,n)$-equipartitions:
for $2 \leq n \leq k$,
\begin{equation}\label{eq:geometric-lower}
 \Delta^*(m,k,n) \geq R(m,k,n)
 := \left\lceil
       \frac{m\alpha_n(k)}{k}+\frac{k-1}{2}
    \right\rceil,
 \qquad
 \alpha_n(k):=\sum_{j=1}^{n}\binom{k}{j}.
\end{equation}
Here $\alpha_n(k)$ is the number of independent equipartition
conditions for each measure, while orthogonality imposes
$\binom{k}{2}$ additional conditions.

Musin also gave a sufficient polynomial condition for the
existence of these partitions. In $\bb{F}_2[x_1,\dots,x_k]$, put
\begin{equation}\label{eq:partition-polynomial}
 Q_{k,j}:=
 \prod_{\substack{I\subseteq\{1,\dots,k\}\ |I|=j}}
 \left(\sum_{i\in I}x_i\right),
 \qquad
 P_{k,n}:=\prod_{j=1}^{n}Q_{k,j}.
\end{equation}
Thus $\deg P_{k,n}=\alpha_n(k)$.

\begin{theorem}[Musin {\cite[Theorem 1.3]{Musin2026}}]
\label[theorem]{thm:musin}
Let $m\geq1$ and $1\leq n\leq k\leq d$. If
\begin{equation}\label{eq:musin-criterion}
 P_{k,n}^{\,m}\notin
 (x_1^{d+1},x_2^d,\dots,x_k^{d-k+2}),
\end{equation}
then every collection of $m$ measures in $\bb{R}^d$ admits an
orthogonal $(k,n)$-equipartition. In particular,
$\Delta^*(m,k,n)\leq d$.
\end{theorem}

We study this criterion for $(k,n)=(4,3)$. For the algebraic
argument, it is convenient to use the symmetric truncation
and define
\begin{equation}\label{eq:algebraic-threshold}
 d^*(m,4,3):=
 \min\left\{
 d:\;
 P_{4,3}^{\,m}\notin
 (x_1^{d+1},x_2^{d+1},x_3^{d+1},x_4^{d+1})
 \right\}.
\end{equation}
Equivalently, $d^*(m,4,3)$ is the least possible largest exponent
of a monomial occurring with nonzero coefficient in
$P_{4,3}^{\,m}$. We show in \Cref{cor:truncation} that, in this
case, the symmetric truncation is equivalent to
\eqref{eq:musin-criterion}. Consequently,
\begin{equation}\label{eq:geometric-algebraic-bounds}
 R(m,4,3)\leq\Delta^*(m,4,3)\leq d^*(m,4,3).
\end{equation}
Since $\alpha_3(4)=14$, the lower bound is
\[
 R(m,4,3)=\left\lceil\frac{7m}{2}+\frac32\right\rceil.
\]

Our main result determines $d^*(m,4,3)$ for every $m$.

\begin{theorem}\label[theorem]{thm:main}
\begin{enumerate}[label=(\roman*)]
 \item For every $j\geq1$,
 \[
  d^*(2^j-1,4,3)
  =R(2^j-1,4,3)
  =7\cdot2^{j-1}-2.
 \]

 \item If
 \[
  m=2^{q+1}(2^j-1)+r,
  \qquad
  j\geq1,\quad q\geq0,\quad 0\leq r\leq2^q-1,
 \]
 then
 \[
  d^*(m,4,3)
  =2^{q+1}\bigl(7\cdot2^{j-1}-2\bigr)+r.
 \]

 \item The equality $d^*(m,4,3)=R(m,4,3)$ holds if and only if
 $m=2^j-1$ for some $j\geq1$.
\end{enumerate}
\end{theorem}

Every positive integer $m$ either equals $2^j-1$ or has a unique
representation as in part~(ii). In the latter case, $q$ is the
position of the highest zero digit in the binary expansion of
$m$, with the least significant position numbered zero; $j$ is
the number of leading ones, and $r$ is formed by the digits
below that zero. Thus \Cref{thm:main} gives the exact threshold
of Musin's criterion for every $m$.

Combining \Cref{thm:main} with
\eqref{eq:geometric-algebraic-bounds} yields the following
geometric consequence.

\begin{corollary}\label[corollary]{cor:geometric-dyadic}
For every $j\geq1$,
\[
 \Delta^*(2^j-1,4,3)=7\cdot2^{j-1}-2.
\]
\end{corollary}

The proof of the main theorem uses symmetry to represent monomials by integer
partitions. We first show that every monomial of
$P_{4,3}^{\,m}$ has pairwise distinct exponents; at the lower
bound, this forces a staircase partition when $m=2^j-1$. Then, using the recurrence relation $P_{4, 3}^{2^{j + 1} - 1} = (P_{4, 3}^{2^j - 1})^2P_{4, 3}$, we prove by induction that the monomial with staircase exponents occurs in $P_{4, 3}^{2^j - 1}$ for any $j \geq 1$. Later we reduce the general case to $m = 2^{q + 1}(2^j - 1) + 2^q - 1$ with $q \geq 1$, $j \geq 1$; we then construct a monomial type at which the minimum is attained and prove that it occurs in $P_{4, 3}^m$ for any $m$ of the type described above.

\section{Algebraic preliminaries}

We use standard notation for partitions and symmetric
polynomials; see \cite[Chapter~I]{Macdonald1995}.

For $\alpha=(\alpha_1,\dots,\alpha_k)\in\bb{Z}_{\geq0}^k$, write
\[
 x^\alpha:=x_1^{\alpha_1}\cdots x_k^{\alpha_k},
 \qquad
 |\alpha|:=\alpha_1+\cdots+\alpha_k.
\]
For $F\in\bb{F}_2[x_1,\dots,x_k]$, let $[x^\alpha]F$ denote
the coefficient of $x^\alpha$, and put
\[
 \supp F:=\{\alpha\in\bb{Z}_{\geq0}^k:[x^\alpha]F=1\}.
\]

A \emph{partition (strict partition) with at most $k$ parts} is a vector
$\lambda\in\bb{Z}_{\geq0}^k$ with
$\lambda_1\geq\cdots\geq\lambda_k$ ($\lambda_1 > \cdots >\lambda_k$); trailing zeros are omitted
when convenient. Its size is $|\lambda|$. Define the
\emph{monomial symmetric polynomial}
\[
 m_\lambda:=\sum_{\alpha\in S_k\lambda}x^\alpha,
\]
where $S_k\lambda$ is the set of distinct permutations of
$\lambda$. The polynomials $m_\lambda$ form a basis of the
symmetric polynomials over $\bb{F}_2$. Thus every symmetric
polynomial has a unique expansion
\[
 F=\sum_\lambda a_\lambda m_\lambda,
 \qquad a_\lambda\in\bb{F}_2.
\]
Hence the $S_k$-orbits of monomials of $F$ correspond to the partitions
$\lambda$ with $a_\lambda=1$.
\begin{lemma}\label[lemma]{lem:vdm}
For $k\geq2$,
\[
 Q_{k,1}Q_{k,2}=m_{(k,k-1,\dots,1)}.
\]
\end{lemma}
\begin{proof}
Over $\bb{F}_2$, the Vandermonde identity gives
$Q_{k,2}=m_{(k-1,k-2,\dots,0)}$.
Multiplication by $Q_{k,1}=x_1\cdots x_k$ yields the claim.
\end{proof}
\begin{lemma}\label[lemma]{lem:distinct}
Let $\lambda$ be a strict partition and $\pi$ a partition. Then every monomial of $m_\lambda m_\pi$ has pairwise distinct exponents; that is, $m_\lambda m_\pi$ is a sum of $m_\theta$ with $\theta$ a strict partition.
\end{lemma}
\begin{proof}
For a partition $\gamma \in \bb{Z}_{\geq0}^k$, $m_\gamma = \sum\limits_{\alpha \in S_k\gamma}x^\alpha$ as defined above, hence
\[
 m_\lambda m_\pi = \sum_{\alpha \in S_k\lambda}\sum_{\beta \in S_k\pi}x^{\alpha + \beta},
\]
and over $\bb{F}_2$ the coefficient of $x^\theta$ is the parity of the number of elements of
\[
 N(\theta) := \bigl\{(\alpha, \beta) \in S_k\lambda \times S_k\pi : \alpha + \beta = \theta\bigr\}.
\]
Let $\theta_i = \theta_j$ for some $i \neq j$, let $\tau \in S_k$ be the transposition of $i$ and $j$, and consider the map $(\alpha, \beta) \mapsto (\tau\alpha, \tau\beta)$.
\begin{enumerate}[label=(\roman*)]
 \item It maps $N(\theta)$ into itself: $\tau\alpha \in S_k\lambda$, $\tau\beta \in S_k\pi$, and $\tau\alpha + \tau\beta = \tau\theta = \theta$ because $\theta_i = \theta_j$.
 \item Applied twice, it is the identity.
 \item It has no fixed points: $\tau\alpha = \alpha$ would mean $\alpha_i = \alpha_j$, while the entries of $\alpha$, a rearrangement of the strict partition $\lambda$, are pairwise distinct.
\end{enumerate}
So $N(\theta)$ splits into pairs $\{(\alpha, \beta), (\tau\alpha, \tau\beta)\}$, $|N(\theta)|$ is even, and $[x^\theta](m_\lambda m_\pi) = 0$. Thus every monomial $x^\theta$ of $m_\lambda m_\pi$ has pairwise distinct exponents. As $m_\lambda m_\pi$ is symmetric, it is the sum of the $m_\theta$ over the partitions $\theta$ with $x^\theta$ a monomial of $m_\lambda m_\pi$, and these $\theta$ are strict partitions.
\end{proof}
\begin{corollary}\label[corollary]{cor:distinct}
Let $k \geq n \geq 2$ and $m \geq 1$. Then every monomial of $P_{k, n}^m$ has pairwise distinct exponents; that is, $P_{k, n}^m$ is a sum of $m_\theta$ with $\theta$ a strict partition.
\end{corollary}
\begin{proof}
Each $Q_{k, j}$ is symmetric, since a permutation of the variables permutes its factors. Write
\[
 P_{k, n}^m = Q_{k, 1}Q_{k, 2}F, \qquad F := Q_{k, 1}^{m - 1}Q_{k, 2}^{m - 1}Q_{k, 3}^m\cdots Q_{k, n}^m,
\]
where $F$ is symmetric, i.e., a sum of some $m_\pi$, and $Q_{k, 1}Q_{k, 2} = m_{(k, \dots, 2, 1)}$ with $(k, \dots, 2, 1)$ a strict partition (\Cref{lem:vdm}). By \Cref{lem:distinct}, every product $m_{(k, \dots, 2, 1)}m_\pi$ has only monomials with pairwise distinct exponents, and hence so does their sum $P_{k, n}^m$.
\end{proof}
\begin{definition}
Throughout the rest of the paper, put $P:=P_{4,3}$.
\end{definition}
\begin{proposition}\label[proposition]{prp:P}
We have
\[
 P=m_{(7,4,2,1)}+m_{(6,5,2,1)}
   +m_{(6,4,3,1)}+m_{(5,4,3,2)}.
\]
Consequently, every
$\pi\in\supp P$ has pairwise distinct entries in $[1,7]$.
\end{proposition}
\begin{proof}
Let $e_i$ denote the elementary symmetric polynomials in four
variables and $s:=e_1$. Then
\[
 Q_{4,3}
 =\prod_{\ell=1}^4(s+x_\ell)
 =s^4+e_1s^3+e_2s^2+e_3s+e_4
 =e_1^2e_2+e_1e_3+e_4.
\]
Since
\[
 e_1^2e_2=m_{(3,1)}+m_{(2,1,1)},
 \qquad
 e_1e_3=m_{(2,1,1)},
\]
we obtain $Q_{4,3}=m_{(3,1)}+m_{(1,1,1,1)}$. Hence (\Cref{lem:vdm})
\[
 P=m_{(4,3,2,1)}
   \bigl(m_{(3,1)}+m_{(1,1,1,1)}\bigr).
\]

To expand $m_{(4,3,2,1)}m_{(3,1)}$, add each of the twelve
permutations of $(3,1,0,0)$ to $(4,3,2,1)$.
After sorting, the vectors with distinct entries give
$(7,4,2,1)$, $(6,5,2,1)$, and $(6,4,3,1)$ once each,
and $(5,4,3,2)$ twice (vectors with repeated entries can be ignored by \Cref{lem:distinct}). Thus
\[
 m_{(4,3,2,1)}m_{(3,1)}
 =m_{(7,4,2,1)}+m_{(6,5,2,1)}+m_{(6,4,3,1)}.
\]
Finally,
$m_{(4,3,2,1)}m_{(1,1,1,1)}=m_{(5,4,3,2)}$.
\end{proof}

In coefficient expressions, we set $[x^\alpha]F=0$
if any entry of $\alpha$ is negative.
\begin{corollary}\label[corollary]{cor:truncation}
For $m\geq1$ and $d\geq4$,
\[
 P^m\notin(x_1^{d+1},x_2^{d+1},x_3^{d+1},x_4^{d+1})
 \quad\Longleftrightarrow\quad
 P^m\notin(x_1^{d+1},x_2^d,x_3^{d-1},x_4^{d-2}).
\]
\end{corollary}
\begin{proof}
For the forward implication, symmetry and
\Cref{cor:distinct} allow us to order a surviving exponent
vector as $\theta_1>\cdots>\theta_4$. Then
$\theta_i\leq d-i+1$, so the corresponding monomial survives
the second truncation. The reverse implication is immediate.
\end{proof}
\begin{definition}
Write $d^*(m):=d^*(m,4,3)$ and
$R_j:=R(2^j-1,4,3)=7\cdot2^{j-1}-2$.
\end{definition}
\begin{proposition}\label[proposition]{thm:lower}
\begin{enumerate}[label=(\roman*)]
 \item $d^*(m)\geq R(m,4,3)$ for every $m\geq1$.
 \item Let $\lambda$ and $\pi = (t, t - 1, t - 2, t - 3)$ be strict partitions with $|\lambda| = |\pi|$ and $\lambda \neq \pi$. Then $\lambda_1 > t$.
\end{enumerate}
\end{proposition}
\begin{proof}
(i) Take $\theta \in \supp P^m$ with $\max_i\theta_i = d^*(m)$. By \Cref{cor:distinct}, after reordering $\theta$ is a strict partition, so $\theta_i \leq \theta_1 - (i - 1)$ and
\[
 14m = |\theta| \leq \sum_{i = 1}^4\bigl(\theta_1 - (i - 1)\bigr) = 4d^*(m) - 6.
\]
Hence $d^*(m) \geq \frac{7m}{2} + \frac32$, and $d^*(m) \geq R(m, 4, 3)$ since $d^*(m)$ is an integer.

(ii) Assume $\lambda_1 \leq t$. Since $\lambda \neq \pi$ and $|\lambda| = |\pi|$, $\lambda$ is not an arithmetic progression with step $-1$ (such a progression is determined by its sum), so $\lambda_{s + 1} < \lambda_s - 1$ for some $s$. Thus $\lambda_i \leq \lambda_1 - (i - 1)$ for all $i$, strictly for $i = s + 1$, and
\[
 |\pi| = \sum_{i = 1}^4\bigl(t - (i - 1)\bigr) \geq \sum_{i = 1}^4\bigl(\lambda_1 - (i - 1)\bigr) > \sum_{i = 1}^4\lambda_i = |\lambda|,
\]
a contradiction.
\end{proof}

\begin{lemma}
Over $\bb{F}_2$, the Frobenius identity gives, for every $F\in\bb{F}_2[x_1,\dots,x_4]$,
\begin{equation}\label{prp:frob}
 F(x_1,\dots,x_4)^2=F(x_1^2,\dots,x_4^2).
\end{equation}
\end{lemma}

\begin{corollary}\label[corollary]{cor:double}
For every $m\geq1$, $d^*(2m)=2d^*(m)$.
\end{corollary}
\begin{proof}
By \eqref{prp:frob},
$\supp P^{2m}=\{2\theta:\theta\in\supp P^m\}$.
\end{proof}

\begin{lemma}\label[lemma]{lem:drop}
For $1\leq m\leq m'$,
\[
 d^*(m)+(m'-m)\leq d^*(m').
\]
\end{lemma}
\begin{proof}
Choose $\theta\in\supp P^{m'}$ with
$\max_i\theta_i=d^*(m')$. Since $P^{m'}=P^mP^{m'-m}$,
there exist $\varphi\in\supp P^m$ and
$\psi\in\supp P^{m'-m}$ such that $\theta=\varphi+\psi$.
As $\psi_i\geq m'-m$,
\[
 d^*(m)\leq\max_i\varphi_i
 \leq d^*(m')-(m'-m).
\]
\end{proof}

\begin{corollary}\label[corollary]{cor:interval}
Let $1\leq a\leq b$. If
$d^*(b)\leq d^*(a)+b-a$, then
\[
 d^*(m)=d^*(a)+m-a
 \qquad\text{for every integer }a\leq m\leq b.
\]
\end{corollary}
\begin{proof}
By \Cref{lem:drop},
\[
 d^*(a)+m-a
 \leq d^*(m)
 \leq d^*(b)-(b-m)
 \leq d^*(a)+m-a.
\]
\end{proof}

\section{Frames and coefficient descent}
\begin{definition}
A \emph{frame} is a sequence $(X_q)$ in $\bb{Z}^4$, indexed by $q\geq0$ or $q\geq1$,
such that
\[
 X_{q+1}=2X_q,
 \qquad
 |X_q|=14\cdot2^q.
\]
A \emph{shape} is a vector $w\in\bb{Z}^4$ with $|w|=14$.
\end{definition}
\begin{definition}
For a fixed frame, put
\[
 c_q(w):=[x^{X_q-w}]P^{2^q-1}.
\]
The degrees agree, since
$|X_q-w|=14(2^q-1)$. At level zero (if $X_0$ is defined),
$c_0(w)=1$ precisely when $w=X_0$.
\end{definition}
\begin{lemma}[descent]\label[lemma]{lem:descent}
For every shape $w$ and $q\geq0$ whenever $X_q$ is defined,
\[
 c_{q+1}(w)
 =\sum_{\substack{\pi\in\supp P\ \pi\equiv w\pmod2}}
 c_q\left(\frac{w+\pi}{2}\right).
\]
\end{lemma}
\begin{proof}
Put $N := 2^q - 1$. By the Frobenius identity (\Cref{prp:frob}),
\[
 P^{2N + 1} = P^N(x_1^2, \dots, x_4^2)\,P = \sum_{\theta \in \supp P^N} \sum_{\pi \in \supp P}x^{2\theta + \pi}.
\]
Since $X_{q + 1} = 2X_q$, the value $c_{q + 1}(w)$ is the number, mod $2$, of pairs $(\theta, \pi)$ in this double sum with $2\theta + \pi = 2X_q - w$ (there are none if $2X_q - w \notin \bb{Z}^4_{\geq0}$). 

For a given $\pi$, such a $\theta$ exists only if $\pi \equiv w \pmod 2$, and then it is unique: $\theta = X_q - v$ with $v := \frac{w + \pi}{2}$. This $v$ is a shape, since $|w| = |\pi| = 14$, and $X_q - v \in \supp P^N$ means exactly $c_q(v) = 1$. Summing over $\pi$ gives the formula.
\end{proof}
\begin{lemma}[symmetry]\label[lemma]{lem:sym}
If $\sigma\in S_4$ fixes the frame, then
$c_q(\sigma w)=c_q(w)$ for every $q\geq0$ whenever $X_q$ is defined.
If moreover $\sigma=(ij)$ and $w_i=w_j$, then
$c_q(w)=0$ for $q\geq1$.
\end{lemma}
\begin{proof}
Put $N := 2^q - 1$ and $\theta := X_q - w$. Since $\sigma X_q = X_q$, we have $X_q - \sigma w = \sigma\theta$, so $c_q(\sigma w) = [x^{\sigma\theta}]P^N = [x^\theta]P^N = c_q(w)$, as $P^N$ is symmetric. 

If $\sigma = (ij)$, then $\sigma X_q = X_q$ means $(X_q)_i = (X_q)_j$, so $w_i = w_j$ gives $\theta_i = \theta_j$, and $c_q(w) = [x^\theta]P^N = 0$ by \Cref{cor:distinct} ($N \geq 1$ as $q \geq 1$).
\end{proof}

\section{The case $m=2^q-1$}\label{sec:dyadic}

For $q\geq1$, consider the frame and shape
\[
 X_q:=2^{q-1}(7,7,7,7),
 \qquad
 \delta:=(2,3,4,5).
\]
Let $c_q$ be its coefficient
function. Then
\[
 X_q-\delta
 =\theta_q
 :=(R_q,R_q-1,R_q-2,R_q-3).
\]
Thus $c_q(\delta)$ is the coefficient of the staircase
partition $\theta_q$.

\begin{theorem}\label[theorem]{thm:AP}
For every $q\geq1$,
\[
 c_q(\delta) = [x^{\theta_q}]P^{2^q-1} = 1.
\]
\end{theorem}
\begin{proof} For $q=1$,
\[
 c_1(\delta)=[x^{(5,4,3,2)}]P=1
\]
by \Cref{prp:P}.

\Cref{lem:descent} gives
\[
 c_{q+1}(\delta)
 =\sum_{\substack{\pi\in\supp P\\
                  \pi\equiv(0,1,0,1)\pmod2}}
 c_q\left(\frac{\delta+\pi}{2}\right).
\]
Each partition in \Cref{prp:P} has two even and two odd
entries, so the sum has $16$ terms, listed below.
The shapes with distinct entries are shown in bold.
\[
\renewcommand{\arraystretch}{1.2}
\begin{array}{c|cc|cc}
 \lambda
 & \pi & (\delta+\pi)/2
 & \pi & (\delta+\pi)/2\\ \hline
 (7,4,2,1)
 & (4,7,2,1) & (3,5,3,3)
 & (4,1,2,7) & (3,2,3,6)\\
 & (2,7,4,1) & \mathbf{(2,5,4,3)}
 & (2,1,4,7) & (2,2,4,6)\\ \hline
 (6,5,2,1)
 & (6,5,2,1) & (4,4,3,3)
 & (6,1,2,5) & \mathbf{(4,2,3,5)}\\
 & (2,5,6,1) & \mathbf{(2,4,5,3)}
 & (2,1,6,5) & (2,2,5,5)\\ \hline
 (6,4,3,1)
 & (6,3,4,1) & (4,3,4,3)
 & (6,1,4,3) & (4,2,4,4)\\
 & (4,3,6,1) & (3,3,5,3)
 & (4,1,6,3) & \mathbf{(3,2,5,4)}\\ \hline
 (5,4,3,2)
 & (4,5,2,3) & (3,4,3,4)
 & (4,3,2,5) & (3,3,3,5)\\
 & (2,5,4,3) & (2,4,4,4)
 & (2,3,4,5) & \mathbf{(2,3,4,5)}
\end{array}
\]
Eleven shapes have repeated entries and vanish by
\Cref{lem:sym}. The remaining five are permutations of
$\delta$. Since $X_q$ is invariant under $S_4$,
\Cref{lem:sym} gives
\[
 c_{q+1}(\delta)=5c_q(\delta)=c_q(\delta).
\]
Hence $c_q(\delta)=1$ for every $q\geq1$.
\end{proof}

\begin{corollary}\label[corollary]{cor:dyadic}
For every $j\geq1$,
\[
 d^*(2^j-1)=R_j=7\cdot2^{j-1}-2.
\]
\end{corollary}
\begin{proof}
By \Cref{thm:AP}, the staircase monomial gives $d^*(2^j-1)\leq R_j$;
the reverse inequality is \Cref{thm:lower}(i).
\end{proof}

\section{The remaining values of $m$}\label{sec:all}

By \Cref{cor:double,cor:dyadic,cor:interval}, it suffices
to construct, for $j,q\geq1$, a monomial of
$P^{2^{q+1}(2^j-1)+2^q-1}$ whose largest exponent is
$2^{q+1}R_j+2^q-1$.
We first introduce another frame, which will help us find such a monomial.

\subsection{A second frame}
Consider the frame
\[
 C_q:=2^q(1,3,5,5),\qquad q\geq0,
\]
and let $c_q$ be its coefficient function. Since $C_q$
is fixed by the transposition $(34)$, \Cref{lem:sym} gives
\begin{equation}\label{eq:C-symmetry}
 c_q(w_1,w_2,w_3,w_4)=c_q(w_1,w_2,w_4,w_3).
\end{equation}
We need only shapes with first entry $1$.

Let $y := (x_2, x_3, x_4)$ and let $\chi_1, \chi_2, \chi_3 \in \bb{Z}^3$ be the unit vectors: $\chi_l$ has $1$ at the place $l$ and $0$ elsewhere. Thus $y^{\chi_l} = x_{l + 1}$, and multiplying a monomial $y^\beta$ by $y^{\chi_l}$ turns its exponent $\beta$ into $\beta + \chi_l$.
\begin{lemma}\label[lemma]{lem:x1}
The monomials of $P$ in which $x_1$ has exponent $1$ are the $18$ monomials $x_1y^{2\sigma + \chi_l}$, where $\sigma$ runs over the permutations of $(1, 2, 3)$ and $l \in \{1, 2, 3\}$. Each vector $2\sigma + \chi_l$ has exactly one odd entry, at the place $l$.
\end{lemma}
\begin{proof}
Since $x_1$ divides $P$, reduction modulo $x_1^2$ keeps exactly the monomials of $P$ with $x_1$-exponent $1$. As $Q_{4, 1} = x_1x_2x_3x_4$, we have $P = x_1F$ with $F := x_2x_3x_4\,Q_{4, 2}Q_{4, 3}$. Let $\bar F$ be the image of $F$ in $\bb{F}_2[x_1, \dots, x_4]/(x_1) \cong \bb{F}_2[y]$; then $F \equiv \bar F \pmod{x_1}$ gives $P \equiv x_1\bar F \pmod{x_1^2}$, and it remains to find $\bar F$. Below, $Q_{3, j}$ stands for $Q_{3, j}(y)$.

The polynomial $Q_{4, 2}$ is the product of six linear forms. Modulo $x_1$, the three forms $x_1 + x_j$ become $x_j$, with product $x_2x_3x_4 = Q_{3, 1}$, while the other three, $x_i + x_j$ with $2 \leq i < j \leq 4$, do not contain $x_1$ and form $Q_{3, 2}$; so $Q_{4, 2} \equiv Q_{3, 1}Q_{3, 2} \pmod{x_1}$. 

Likewise, the three forms $x_1 + x_i + x_j$ of $Q_{4, 3}$ become the forms $x_i + x_j$ of $Q_{3, 2}$, and its fourth form is $x_2 + x_3 + x_4 = Q_{3, 3}$; so $Q_{4, 3} \equiv Q_{3, 2}Q_{3, 3} \pmod{x_1}$. Together with $x_2x_3x_4 = Q_{3, 1}$, this gives $\bar F = Q_{3, 3}(Q_{3, 1}Q_{3, 2})^2 = (x_2 + x_3 + x_4)\,m_{(3, 2, 1)}(y)^2$ by \Cref{lem:vdm}.

By \Cref{prp:frob}, $m_{(3, 2, 1)}(y)^2 = \sum_\sigma y^{2\sigma}$, and multiplication by $x_2 + x_3 + x_4$ adds some $\chi_l$ to the exponent. No terms cancel: $2\sigma + \chi_l$ determines $l$ as the place of its only odd entry, and then $\sigma$.
\end{proof}
\begin{lemma}[descent at the bottom]\label[lemma]{lem:bottom}
Let $q \geq 0$ and $w' \in \bb{Z}^3$ with $|w'| = 13$; then $w'$ has one or three odd entries.
\begin{enumerate}[label=(\alph*)]
 \item If $w'$ has three odd entries, then $c_{q + 1}(1, w') = 0$.
 \item If $w'$ has one odd entry, at the place $i$, and $u := \frac{w' + \chi_i}{2}$ (that is, $u = \lceil w'/2\rceil$, rounded up entrywise), then
 \[
  c_{q + 1}(1, w') = \sum_\sigma c_q(1, u + \sigma),
 \]
 the sum over the six permutations $\sigma$ of $(1, 2, 3)$.
\end{enumerate}
\end{lemma}
\begin{proof}
By \Cref{lem:descent}, $c_{q + 1}(1, w')$ is the sum of $c_q(v)$, $v := \frac{(1, w') + \pi}{2}$, over the $\pi \in \supp P$ with $\pi \equiv (1, w') \pmod 2$. Since $x_1$ divides $P$, $c_q(v) \neq 0$ needs the first entry $2^q - v_1$ of $C_q - v$ to be at least $2^q - 1$, that is, $\pi_1 = 1$. By \Cref{lem:x1} such $\pi$ are $(1, 2\sigma + \chi_l)$, and $2\sigma + \chi_l \equiv \chi_l \pmod 2$.

So if $w'$ has three odd entries, no $\pi$ has the parity of $(1, w')$, which proves (a). If $w' \equiv \chi_i \pmod 2$, exactly the six $\pi$ with $l = i$ have it, and for them $v = (1, u + \sigma)$, which proves (b).
\end{proof}
\begin{corollary}\label[corollary]{cor:pair}
Let $q \geq 0$ and $u \in \bb{Z}^3$ with $|u| = 7$. If the last two entries of $u$ are equal, then
\[
 c_{q+1}(1,2u-\chi_i)=0,
 \qquad i=1,2,3.
\]
\end{corollary}
\begin{proof}
The vector $2u - \chi_i$ has its only odd entry at the place $i$, and $\frac{(2u - \chi_i) + \chi_i}{2} = u$. So by \Cref{lem:bottom}(b),
\[
 c_{q + 1}(1, 2u - \chi_i) = \sum_\sigma c_q(1, u + \sigma).
\]
Let $\tau$ exchange the last two entries of a vector in $\bb{Z}^3$. Since $\tau u = u$, we have $\tau(u + \sigma) = u + \tau\sigma$, so $\tau$ permutes the six vectors $u + \sigma$; it fixes none of them, because the last two entries of $\sigma$ differ. Hence the six terms split into three pairs $c_q(1, u + \sigma)$, $c_q(1, u + \tau\sigma)$, and the two terms of each pair are equal by \eqref{eq:C-symmetry}. So the sum is even, that is, $0$.
\end{proof}

\begin{corollary}\label[corollary]{cor:two}
For every $q\geq1$,
\[
 c_q(1,3,6,4)=1,
 \qquad
 c_q(1,6,3,4)=0.
\]
Consequently,
\[
 \nu_q:=
 (2^q-1,\;3\cdot2^q-3,\;5\cdot2^q-6,\;5\cdot2^q-4)
 \in\supp P^{2^q-1},
\]
whereas $C_q-(1,6,3,4)\notin\supp P^{2^q-1}$.
\end{corollary}
\begin{proof}
For $w'=(6,3,4)$, \Cref{lem:bottom} gives $u=(3,2,2)$,
so \Cref{cor:pair} yields $c_q(1,6,3,4)=0$ for $q\geq1$.

For the first coefficient, put $u:=(2,3,2)$ and
\[
 s_q:=c_{q + 1}(1, 3, 6, 4) = \sum_\sigma c_q(1,u+\sigma).
\]
The six vectors $u+\sigma$ are
\[
 (3,5,5),\quad(5,5,3),\quad(5,4,4),\quad
 (3,6,4),\quad(4,6,3),\quad(4,4,5).
\]
At level zero only $(3,5,5)$ contributes, so $s_0=1$.
For $q\geq1$, the first two coefficients vanish by
\Cref{lem:bottom}, and the third by \Cref{cor:pair}.
Hence
\[
 s_q=c_q(1,3,6,4)+c_q(1,4,6,3)+c_q(1,4,4,5).
\]
By \Cref{lem:bottom},
\[
 c_q(1,3,6,4)
 =c_q(1,4,5,4)
 =c_q(1,4,6,3)
 =s_{q-1}.
\]
Also $c_q(1,4,4,5)=c_q(1,4,5,4)$ by
\eqref{eq:C-symmetry}. Thus $s_q=3s_{q-1}=s_{q-1}$,
and $s_q=1$ for all $q\geq0$.
Finally, $c_q(1,3,6,4)=s_{q-1}=1$ for $q\geq1$.
\end{proof}

\subsection{The endpoint monomial}

\begin{theorem}\label[theorem]{thm:T}
Let $j,q\geq1$, and put
\[
 m:=2^{q+1}(2^j-1)+2^q-1,\qquad
 \theta^0:=(R_j,R_j-1,R_j-2,R_j-3),
 \qquad
 T:=2^{q+1}\theta^0+\nu_q.
\]
Then $T\in\supp P^m$ and
\[
 \max_iT_i=T_1=2^{q+1}R_j+2^q-1.
\]
\end{theorem}
\begin{proof}
By the definition of $\nu_q$,
\[
 T=2^{q+1}R_j(1,1,1,1)
   +(2^q-1,\;2^q-3,\;2^q-6,\;-2^q-4),
\]
so $T_1$ is the largest entry.

By Frobenius,
\[
 P^m
 =P^{2^j-1}(x_1^{2^{q+1}},\dots,x_4^{2^{q+1}})
  P^{2^q-1}.
\]
Thus $[x^T]P^m$ is the parity of the number of pairs
\[
 \theta\in\supp P^{2^j-1},\qquad
 \lambda\in\supp P^{2^q-1},\qquad
 2^{q+1}\theta+\lambda=T.
\]
Since $\lambda_i\geq2^q-1$ and
$T_i\leq2^{q+1}R_j+2^q-1$, every $\theta_i\leq R_j$.
By symmetry, \Cref{cor:distinct} and \Cref{thm:lower}(ii),
$\theta$ is a permutation of $\theta^0$.
Write
\[
 a :=\sigma(0,1,2,3), \sigma \in S_4,\qquad
 \theta=R_j(1,1,1,1)-a.
\]
Then
\[
 \lambda=\nu_q+2^{q+1}\bigl(a -(0,1,2,3)\bigr).
\]
The bounds $2^q-1\leq\lambda_i\leq7(2^q-1)$ fix the places of $0$ and $3$ in $a$. If $a_i=0$ for some $i\geq2$, then $\lambda_i$ is $2^q-3$, $2^q-6$, or $-2^q-4$, below the lower bound; so $a_1=0$. If $a_i=3$ for some $i\leq3$, then $\lambda_i$ is $7\cdot2^q-1$, $7\cdot2^q-3$, or $7\cdot2^q-6$, above the upper bound; so $a_4=3$. Hence $a=(0,1,2,3)$ or $a=(0,2,1,3)$.

In the second case,
\[
 \lambda=
 (2^q-1,\;5\cdot2^q-3,\;3\cdot2^q-6,\;5\cdot2^q-4).
\]
This is $C_q-(1,6,3,4)$ with its middle entries exchanged,
so it does not occur by \Cref{cor:two} and symmetry.
The only contributing pair is therefore
$(\theta^0,\nu_q)$, whose two factors occur by
\Cref{thm:AP} and \Cref{cor:two}. Thus $[x^T]P^m=1$.
\end{proof}

\subsection{Proof of the main theorem}

\begin{proof}[Proof of \Cref{thm:main}]
Part~(i) is \Cref{cor:dyadic}.

For part~(ii), put
\[
 a:=2^{q+1}(2^j-1),\qquad b:=a+2^q-1.
\]
By \Cref{cor:double,cor:dyadic},
$d^*(a)=2^{q+1}R_j$.
If $q=0$, then $r=0$ and the claim follows.
For $q\geq1$, \Cref{thm:T} gives
\[
 d^*(b)\leq2^{q+1}R_j+2^q-1=d^*(a)+b-a.
\]
Thus \Cref{cor:interval} yields
\[
 d^*(a+r)=2^{q+1}R_j+r,
 \qquad 0\leq r\leq2^q-1.
\]

For part~(iii), equality holds when $m=2^j-1$ by part~(i).
Otherwise, write $m=2^{q+1}(2^j-1)+r$ as in part~(ii).
Then
\[
 \begin{aligned}
 d^*(m)-R(m,4,3)
 &=3\cdot2^q+r-\left\lceil\frac{7r+3}{2}\right\rceil\\
 &=3\cdot2^q-\left\lceil\frac{5r+3}{2}\right\rceil\\
 &\geq3\cdot2^q-\frac{5r+4}{2}\\
 &\geq\frac{2^q+1}{2}>0,
 \end{aligned}
\]
where the penultimate inequality uses $r\leq2^q-1$.
\end{proof}

\noindent {\bf {Acknowledgments.}}
Platon Khamatov and Alexey Fakhrutdinov would like to express their sincere gratitude to the organizers of
the Research Experience Program for Undergraduates 2026 (LIPS) for providing an excellent research platform and continuous support. The program
was hosted by the Laboratory of Combinatorial and Geometric Structures
at the Phystech School of Applied Mathematics and Informatics, MIPT.

\newpage

 A. A. Fakhrutdinov, MIPT Centre of Pure Mathematics.

 {\it E-mail address:} fakhrutdinov.aa@phystech.edu

\medskip

 P. A. Khamatov, MIPT School of Applied Mathematics and Computer Science.

 {\it E-mail address:} khamatov.pa@phystech.edu

\medskip

 O. R. Musin,  University of Texas Rio Grande Valley, School of Mathematical and
 Statistical Sciences, One West University Boulevard, Brownsville, TX, 78520, USA.

 {\it E-mail address:} oleg.musin@utrgv.edu

\end{document}